\documentclass[12pt,a4paper]{amsart}
\usepackage{amsfonts}
\usepackage{amsthm}
\usepackage{amsmath}
\usepackage{amscd}
\usepackage[latin2]{inputenc}
\usepackage{t1enc}
\usepackage[mathscr]{eucal}
\usepackage{indentfirst}
\usepackage{graphicx}
\usepackage{graphics}

\numberwithin{equation}{section}

\theoremstyle{plain}
\newtheorem{Th}{Theorem}[section]

\newtheorem{Cor}[Th]{Corollary}

 \theoremstyle{definition}

\newtheorem{?}[Th]{Problem}

\DeclareMathSizes{5}{5}{5}{5}

\begin{document}

\title{Christoffel-Darboux type identities for independence polynomial}

\author{Ferenc Bencs} 
 \address{E\"otv\"os Lor\'and Tudom\'anyegyetem
 \\ H-1117 Budapest
 \\ P\'azm\'any P\'eter s\'et\'any 1/C \\ Hungary} 
 \email{ferenc.bencs@gmail.com}



\begin{abstract} In this paper we introduce some Christoffel-Darboux type
 identities for independence polynomials. As an application, we give a new
 proof of a theorem of  M. Chudnovsky and P. Seymour, claiming that the independence
 polynomial of a claw-free graph has only real roots. Another application is related to a conjecture of Merrifield and Simmons.
\end{abstract}

\maketitle

\section{Introduction} 
The independence polynomial of a graph $G$ is defined by:
\small
\[
	I(G,x)=\sum\limits_{k=0}a_k(G)x^k,
\]
where $a_0(G)=1$, and if $k\ge 1$, then $a_k(G)$ denotes the number of independent sets of $G$ of size $k$.
The matching polynomial of a graph $G$ is defined in a similar way:
\[
	\mu(G,x)=\sum\limits_{k=0}(-1)^km_k(G)x^{n-2k},
\]
where $m_0(G)=1$, and if $k\ge 1$, then $m_k(G)$ is the number of matchings of size $k$.
\medskip

The Christoffel-Darboux identity is one of the most important tools in the theory of orthogonal polynomials. It asserts, that if $(p_n(x))$ is a sequence of orthogonal polynomials, then 
$$\sum_{j=0}^n \frac{1}{h_j}p_j(x)p_j(y)=\frac{k_n}{h_nh_{n+1}}\frac{p_n(y)p_{n+1}(x)-p_n(x)p_{n+1}(x)}{x-y},$$
where $h_j$ is the squared norm of $p_j(x)$, and $k_j$ is the leading coefficient of $p_j(x)$. A one-line consequence of this identity is the real-rootedness of the polynomial $p_n(x)$ for any $n$. Indeed, assume that $\xi$ is a non-real root of $p_n(x)$, and let $x=\xi$ and $y=\overline{\xi}$. Then the right hand side is $0$, while the left hand side is positive, since $\frac{1}{h_j}p_j(\xi)p_j(\overline{\xi})$ is nonnegative for all $j$, and for $j=1$, this term is positive.

When Heilmann and Lieb \cite{hei} introduced the theory of matching polynomials, they already noticed that the matching polynomials show strong analogies with the orthogonal polynomials. They proved the following Christoffel-Darboux identities:

\begin{Th} \label{M1}
Let $G$ be a graph and $u,v\in V(G)$. Let $\mathcal{P}_{u,v}$ be the set of paths from $u$ to $v$. Then
\[
	\mu(G-u,x)\mu(G-v,x)-\mu(G,x)\mu(G-u-v,x)=\sum\limits_{P\in \mathcal{P}_{u,v}}\mu(G-P,x)^2
\]
\end{Th}
\begin{Th} \label{M2}
Let $G$ be a graph and $u\in V(G)$. Let $\mathcal{P}_u$ be the set of paths starting from $u$. Then
$$ \mu(G,x)\mu(G-u,y)-\mu(G-u,x)\mu(G,y)=$$
$$= (x-y)\sum\limits_{P\in \mathcal{P}_u}\mu(G-P,x)\mu(G-P,y)$$
\end{Th}

Note that Theorem~\ref{M2} provides a fast proof of the fact, that all matching polynomials have only real zeros. The argument is almost the same as the one given above for orthogonal polynomials.
\medskip
 
The aim of this paper is to extend these identities to independence polynomials, and to give a new proof of a theorem of P. Seymour and M. Chudnovsky, which claims that the independence polynomial of a claw-free graph has only real roots. A graph is claw-free if it does not contain and induced $K_{1,3}$, the complete bipartite graph on $1+3$ vertices. This is a generalization of the Heilmann-Lieb theorem, which claims that the matching polynomial has only real zeros, since every line graph is claw-free, and the independence polynomial of a line graph is essentially the matching polynomial of the original graph up to some simple transformations. 

To introduce our Christoffel-Darboux type identites for the independence polynomials, we need some notations. We denote the vertex set
and edge set of a graph $G$ by $V(G)$ and $E(G)$, respectively. Let $N_G(u)$
denote the set of neighbours of the vertex $u$.  Let $G-e$ denote the graph obtained from $G$ by deleting the
edge $e$. For $H\subseteq V(G)$, let $N[H]=N_G(H)\cup H$ be the closure of $H$. If $S\subseteq V(G)$, then $G-S$ denotes the induced subgraph of $G$  on the vertex set $V(G)\setminus S$. For a graph $G$ and $u,v\in V(G)$, let $d_G(u,v)$ denote the length of the shortest walk from $u$ to $v$ in $G$, if it exists, else let it be $\infty$.

We will prove the following theorems.

\begin{Th} \label{T1} Let $G$ be a graph, and $u,v \in V(G)$. Let $\mathcal{B}_{u,v}$ be the set of induced connected, bipartite graphs containing the vertices $u$ and $v$. Then
$$I(G-u,x)I(G-v,x)-I(G,x)I(G-u-v,x)=$$ 
$$= \sum\limits_{H \in \mathcal{B}_{u,v}} (-1)^{d_H(u,v)+1}x^{|V(H)|}I(G-N[H],x)^2$$
\end{Th}

\begin{Th} \label{T2} Let $G$ be a graph, and $u\in V(G)$. Let $\mathcal{B}_u$ be
 the set of induced connected, bipartite graphs containing the vertex
 $u$. For an $H\in \mathcal{B}_u$, let $A(H)$ be the color class
 containing $u$, and let $B(H)$ be the color class
 not containing $u$, and $|A(H)|=a(H)$ and $|B(H)|=b(H)$. Then

$$I(G,x)I(G-u,y)-I(G-u,x)I(G,y)=$$
$$=\sum_{H\in \mathcal{B}_u}I(G-N[H],x)I(G-N[H],y)(x^{a(H)}y^{b(H)}-x^{b(H)}y^{a(H)}).$$
\end{Th}

All four theorems can be proved with the idea of examining two ``interesting'' (matching or independent set) sets in $G$. In the case of matchings we see cycles and paths, and  in the case of independence sets we see double points and bipartite graphs.


A special case of Theorem 1.3 is related to the so called Merrifield-Simmons conjecture. This conjecture asserts, that for every graph $G$ and $u,v\in V(G)$, the sign of $I(G-u,1)I(G-v,1)-I(G,1)I(G-\{u,v\},1)$ depends only on the parity of the distance of $u$ and $v$ in $G$.
This was claimed to be true without proof  in their book \cite{merrifield}, and became known as the Merrifield-Simmons conjecture.
This conjecture turned out to be false for general graphs, as it was pointed out in \cite{gutman}. On the other hand, it turned out, that the conjecture is true for bipartite graphs \cite{martin}. Now we see that Theorem \ref{T1} implies a slight generalization of this result:
Suppose that  every induced path in $G$ form $u$ to $v$ has the same parity of length, then $(-1)^{d_G(u,v)}=(-1)^{d_H(u,v)}$ for every $H \in \mathcal{B}_{u,v}$. Especally, when $G$ is bipartite, then for every $u,v\in V(G)$ the parity of all pathes from $u$ to $v$ are the same. 

\begin{Cor}
Let $G$ be a bipartite graph, and $u,v\in V(G)$ and $x\in \mathbb{R}^+$. Then
\begin{eqnarray*}
I(G-u,x)I(G-v,x)-I(G,x)I(G-u-v,x)>0 ~ \textrm{if $d_G(u,v)$ is odd}\\
I(G-u,x)I(G-v,x)-I(G,x)I(G-u-v,x)=0 ~ \textrm{if $d_G(u,v)=\infty$}\\
I(G-u,x)I(G-v,x)-I(G,x)I(G-u-v,x)<0 ~ \textrm{if $d_G(u,v)$ is even}.\\
\end{eqnarray*}
\end{Cor}

Corollary~\ref{C1} and Theorem~\ref{T3} are consequences of the Theorem~\ref{T1} and \ref{T2} using the facts that 
\[
	I'(G,x)=\sum\limits_{u\in V(G)}I(G-N[u],x)
\] and 
\[
	I(G,x)=I(G-u,x)+xI(G-N[u],x).
\]
The proof of Corollary~\ref{C1} can be found in \cite{god3} for matching polynomials and is quite similar for independence polynomials, therefore we will not give  the  detailed proof of it. 

\begin{Cor}\label{C1}
Let $G$ be a graph and $u\in V(G)$. Let $\mathcal{B}_u$ be
 the set of induced connected, bipartite graphs containing the vertex
 $u$. For an $H\in \mathcal{B}_u$ let $A(H)$ be the color class
 containing $u$, and let $B(H)$ be the color class
 not containing $u$, and $|A(H)|=a(H)$ and $|B(H)|=b(H)$. Then
$$xI'(G-u,x)I(G,x)-xI(G-u,x)I'(G,x)=$$
$$=\sum\limits_{H \in B_u} (b(H)-a(H))x^{|V(H)|}I(G-N[H],x)^2.$$
Let $\mathcal{B}$ be the set of induced connected, bipartite graphs. For an $H\in \mathcal{B}$ let $P(H)$ be one of the color classes, and $R(H)$ the other, and let $p(H)=|P(H)|$ and $r(H)=|R(H)|$. Then
\begin{eqnarray*}
x^2I'(G,x)^2-x^2I''(G,x)I(G,x)-xI'(G,x)I(G,x)=\\
=-\sum\limits_{H\in\mathcal{B}}(p(H)-r(H))^2x^{|V(H)|}I(G-N[H],x)^2
\end{eqnarray*}
 \end{Cor}
 
 In the proof of  Theorem~\ref{T3} we will follow an argument similar to the one given in \cite{god3} for matching polynomials.

\begin{Th} \label{T3} Let $G$ be a graph. Let $\mathcal{B}$ be the set of induced connected, bipartite graphs. For an $H\in \mathcal{B}$ let $P(H)$ be one of the color classes, and $R(H)$ the other, and let $p(H)=|P(H)|$ and $r(H)=|R(H)|$. Then
\[
yI(G,x)I'(G,y)-xI'(G,x)I(G,y)=
\]
\[
=\sum_{H\in \mathcal{B}}(p(H)-r(H))I(G-N[H],x)I(G-N[H],y)(x^{p(H)}y^{r(H)}-x^{r(H)}y^{p(H)}).
\]
\end{Th}
By an application of Theorem~$\ref{T3}$ we will prove the following theorem of M. Chudnovsky and P. Seymour.
\begin{Th} \label{T4} The independence polynomial of a claw-free graph has only
 real roots.
\end{Th}

In this paper we prove this theorem using the identity of Theorem~\ref{T2}. Another similar proof of this theorem can be found in \cite{bodo2}. 

This paper is organized as follows. In the next section we prove
Theorem~\ref{T1}, Theorem~\ref{T2} and Theorem~\ref{T3}. In the third
section we prove Theorem~\ref{T4}.

\section{Proof of the Christoffel-Darboux identities}
\begin{proof}[Proof of Theorem~\ref{T1}]
Let $\mathcal{F}(G)$ be the set of the independent sets of $G$. Let $\mathcal{F}_1=\mathcal{F}(G-u)\times\mathcal{F}(G-v)$, and $\mathcal{F}_2=\mathcal{F}(G)\times \mathcal{F}(G-u-v)$.  By definition  $I(G)$ is equal to $\sum\limits_{A\in\mathcal{F}(G)}x^{|A|}$. Then the left hand side of the identity yields:
\[
\left(\sum\limits_{(A,B)\in \mathcal{F}_1}x^{|A|+|B|}\right)- \left(\sum\limits_{(A,B)\in \mathcal{F}_2}x^{|A|+|B|}\right)
\]
We have to understand the relation of two independent sets in $G$. Let $X\subset V(G)$, then let $G[X]$ be the induced graph in $G$. Then let $A$ and $B$ be independent sets of $G$. It follows easily, that if $z\in A\cap B$ than $d_{G[A\cup B]}(z)=0$ and $G[A\bigtriangleup B]$ is a bipartite graph where the color classes are $A-B$ and $B-A$. 

Suppose that $(A,B)\in\mathcal{F}_1$ and $u,v\notin A\cup B$. Then $(A,B)\in\mathcal{F}_2$, because $A,B\in \mathcal{F}(G-u-v)\subset \mathcal{F}(G)$, and it is true vica versa, if $(C,D)\in\mathcal{F}_2$ and $u,v\notin C\cup D$, then $(C,D)\in\mathcal{F}_1$. Therefore there exists a bijection between $\{(A,B)\in \mathcal{F}_1~|~u,v\notin A\cup B\}$ and $\{(A,B)\in \mathcal{F}_2~|~u,v\notin A\cup B\}$, and these terms cancel each other.

Suppose that $(A,B)\in \mathcal{F}_1$ and $v\in A$ and $u\notin B$. Then $A\in \mathcal{F}(G)$ and $B\in \mathcal{F}(G-u-v)$, so $(A,B)\in \mathcal{F}_2$. 

Suppose that $(A,B)\in \mathcal{F}_1$, and $v\notin A$ and $u\in B$. Then $(B,A)\in \mathcal{F}_2$.

Let $\mathcal{F}_1'=\{(A,B)\in \mathcal{F}_1 ~|~ u,v\in A\cup B\}$, and $\mathcal{F}_2'=\{(A,B)\in \mathcal{F}_2 ~|~ u,v\in A\cup B\}$. After simplifying the formula with the bijections we get
\[
\left(\sum\limits_{(A,B)\in \mathcal{F}_1'}x^{|A|+|B|}\right)- \left(\sum\limits_{(A,B)\in \mathcal{F}_2'}x^{|A|+|B|}\right).
\]


Now we see that $u$ and $v$ are always in the color classes, and they are never in $A\cap B$. This means that they are always part of a bipartite graph. 

Suppose that $(A,B)\in \mathcal{F}_1'$, and $u$ and $v$ are not in the same component of $G[A\cup B]$. After switching the colors only in the component of $v$, let the new independent sets be $A'$ and $B'$. Then $u,v \in A' \in\mathcal{F}(G)$ and $u,v\notin B'\in \mathcal{F}(G)$, so $(A',B')\in \mathcal{F}_2'$ and $|A|+|B|=|A'|+|B'|$. It is easy to see that every pair is cancelled, where $u$ and $v$ are not in the same component. Let $\mathcal{F}_1''=\{(A,B)\in \mathcal{F}_1' ~|~ d_{G[A\cup B]}(u,v)<\infty \}$, and $\mathcal{F}_2''=\{(A,B)\in \mathcal{F}_2' ~|~ d_{G[A\cup B]}(u,v)<\infty \}$. We can rewrite the left hand side as
\[
\left(\sum\limits_{(A,B)\in \mathcal{F}_1''}x^{|A|+|B|}\right)- \left(\sum\limits_{(A,B)\in \mathcal{F}_2''}x^{|A|+|B|}\right).
\]
Let us observe, that if $(A,B)\in \mathcal{F}_1''$, then $d_{G[A\cup B]}(u,v)$ is odd, and if $(A,B)\in \mathcal{F}_2''$, then $d_{G[A\cup B]}(u,v)$ is even. So if $A,B$ are independent sets of $G$, their union contains $u$ and $v$, and they are in the same component of the induced graph, then it is easy to recognize whether $(A,B)\in \mathcal{F}_1''$ or $(A,B)\in \mathcal{F}_2''$. Let $(A,B)\in \mathcal{F}_1''\cup \mathcal{F}_2''$, then let $P(A,B)$ be the connected component of $u$ and $v$ in the induced graph. Then
\begin{eqnarray*}
\left(\sum\limits_{(A,B)\in \mathcal{F}_1''\cup\mathcal{F}_2''}(-1)^{d_{P(A,B)}(u,v)+1}x^{|P(A,B)|}x^{|A|+|B|-|P(A,B)|}\right)=\\ 
=\sum\limits_{H \in \mathcal{B}_{u,v}}  (-1)^{d_H(u,v)+1}x^{|V(H)|}\left(\sum\limits_{A,B\in\mathcal{F}(G-N[H])}x^{|A|+|B|} \right)=\\
=\sum\limits_{H \in \mathcal{B}_{u,v}} (-1)^{d_H(u,v)+1}x^{|V(H)|}I(G-N[H],x)^2
\end{eqnarray*}
\end{proof}
\begin{proof}[Proof of Theorem~\ref{T2}]
We will use the same argument as in the previous proof. Let $\mathcal{F}(G)$ be the set of independent sets and let $\mathcal{F}_1=\mathcal{F}(G)\times \mathcal{F}(G-u)$ and let $\mathcal{F}_2=\mathcal{F}(G-u)\times \mathcal{F}(G)$. Then the left hand side is equal to
\begin{eqnarray*}
\sum\limits_{(A,B)\in \mathcal{F}_1} x^{|A|}y^{|B|} - \sum\limits_{(A,B)\in \mathcal{F}_2} x^{|A|}y^{|B|}
\end{eqnarray*}

Suppose that $(A,B)\in \mathcal{F}_1$ and $u\notin A$ and $(C,D)\in \mathcal{F}_2$ and $u\notin D$, then $(A,B)\in \mathcal{F}_2$ and $(C,D)\in \mathcal{F}_1$, so there is a cancellation. Let $\mathcal{F}=\mathcal{F}_1\setminus \mathcal{F}(G-u)\times \mathcal{F}(G-u)$. Then the left hand side is equal to
\[
\sum\limits_{(A,B)\in \mathcal{F}} x^{|A|}y^{|B|}-\sum\limits_{(A,B)\in \mathcal{F}} x^{|B|}y^{|A|}.
\]

Note that for all $(A,B) \in \mathcal{F}$, $u$ always in $A$ and $u\notin A\cap B$. Let $P(A,B)$ be the connected component of the graph induced by the set by $A\cup B$ which contains $u$. Then we can write the following.
\begin{eqnarray*}
\sum\limits_{(A,B)\in \mathcal{F}} x^{|A|}y^{|B|}=\\
=\sum\limits_{(A,B)\in \mathcal{F}} x^{a(P(A,B))}y^{b(P(A,B))} x^{|A|-a(P(A,B))}y^{|B|-b(P(A,B))} =\\
=\sum\limits_{H\in B_u}x^{a(H)}y^{b(H)}\sum\limits_{K,L\in \mathcal{F}(G-[H])}x^{|K|}y^{|L|}=\\
=\sum\limits_{H\in B_u}x^{a(H)}y^{b(H)}I(G-N[H],x)I(G-N[H],y)
\end{eqnarray*}
We get the same formula for the second sum.
\begin{eqnarray*}
\sum\limits_{(A,B)\in \mathcal{F}} x^{|B|}y^{|A|}=\\
=\sum\limits_{H\in B_u}y^{a(H)}x^{b(H)}I(G-N[H],y)I(G-N[H],x)
\end{eqnarray*}
The left hand side of the identity is equal to
\begin{eqnarray*}
=\sum\limits_{H\in B_u}(x^{a(H)}y^{b(H)}-y^{a(H)}x^{b(H)})I(G-N[H],x)I(G-N[H],y)
\end{eqnarray*}
\end{proof}
\begin{proof}[Proof of Theorem~\ref{T3}]
We use the facts that \[I'(G,x)=\sum\limits_{u\in V(G)}I(G-N[u],x) \]
and \[I(G,x)=I(G-u,x)+xI(G-N[u],x).\]
Let $n=|V(G)|$, and by combinating these two formulae we get
\[
\sum\limits_{u\in V(G)}I(G-u,x)=nI(G,x)-xI'(G,x).
\]
Let us sum the identity of Theorem \ref{T2} for all $u\in V(G)$ and apply the above identities. 
\begin{eqnarray*}
\sum\limits_{u\in V(G)}\left( I(G,x)I(G-u,y)-I(G-u,x)I(G,y)\right)=\\
=I(G,x)\sum\limits_{u\in V(G)}I(G-u,y) - I(G,y)\sum\limits_{u\in V(G)}I(G-u,x)=\\
=I(G,x)(nI(G,y)-yI'(G,y))-I(G,y)(nI(G,x)-xI'(G,x))=\\
=xI'(G,x)I(G,y)-yI(G,x)I'(G,y)=
\end{eqnarray*}
\begin{eqnarray*}
=\sum\limits_{u\in V(G)}\sum\limits_{H\in B_u} (x^{a(H)}y^{b(H)}-y^{a(H)}x^{b(H)})I(G-N[H],x)I(G-N[H],y)=\\
\sum\limits_{H\in\mathcal{B}}(p(H)-r(H))(x^{p(H)}y^{r(H)}-y^{p(H)}x^{r(H)})I(G-N[H],x)I(G-N[H],y)
\end{eqnarray*}
\end{proof}
\section{Claw-free graphs} 
Let us prove Theorem \ref{T4} by applying the identity of Theorem \ref{T3}. First of all, note that every induced connected bipartite subgraph of a claw-free graph is a path or a cycle. Indeed, claw-freeness implies that every degree in an induced bipartite graph is at most 2, connectedness implies that it is a path or a cycle. Let P be a path on even number of vertices, then $p(P)-r(P)=0$. The same holds for a cycle. If $|P|$ is odd, then choose $P(P)$ and $R(P)$ in a way that $p(P)-r(P)=1$.
Let $\mathcal{P}$ be the set of paths in $G$ with odd vertices. Then we get the following.
\[
\frac{yI(G,x)I'(G,y)-xI'(G,x)I(G,y)}{y-x}=\sum\limits_{P\in \mathcal{P}}x^{r(P)}y^{r(P)}I(G-N[P],x)I(G-N[P],y)
\]

Suppose that $G$ is the smallest counterexample, such that $I(G,x)$ has not only real roots. Let us denote one of these roots by $\xi$. Then evaluate this formula with $x=\xi$ and $y=\bar{\xi}$. Thus, the left hand side is 0. But on the right hand side every $G-N[P]$ is an induced graph, so these are also claw-free graphs with less vertices than $G$. All in all, $I(G-N[P],\xi)\neq 0$, and this means that the right hand side is positive, but it gives a contradiction.

\vspace{10pt}\noindent \textbf{Acknowledgments.}
 I would like to express my sincere gratitude to P\'eter Csikv\'ari for introducing me a set of problems which led to this paper. I would also like to thank for his gentle guidance and great effort which helped my work run smoothly. These results would not have been achieved without his supervision.
 
\nocite{god4, hei, bodo}
\bibliography{hivatkozat}
\bibliographystyle{unsrt}

\end{document}